\documentclass[11pt]{article}
\usepackage[margin=1in]{geometry}          
\usepackage{graphicx}
\usepackage{amsthm, amsmath, amssymb}
\usepackage{setspace}
\usepackage[loose,nice]{units} 
\usepackage{tensor}
\usepackage{mathrsfs}
\usepackage{graphicx}
\usepackage{xcolor}
\usepackage[font={small,it}]{caption}
\graphicspath{{/Users/swagatamsen/Documents}}
\usepackage{kpfonts}
\usepackage{enumitem}

\DeclareSymbolFont{matha}{OML}{txmi}{m}{it}
\DeclareMathSymbol{\varv}{\mathord}{matha}{118}
\numberwithin{equation}{section}

\theoremstyle{plain}
\newtheorem{thm}{Theorem}[section]
\newtheorem{lem}[thm]{Lemma}
\newtheorem{prop}[thm]{Proposition}
\newtheorem*{cor}{Corollary}

\theoremstyle{definition}
\newtheorem{defn}{Definition}[section]

\newtheorem{exmp}{Example}[section]

\theoremstyle{remark}
\newtheorem*{rem}{Remark}

\newcommand{\interior}[1]{%
  {\kern0pt#1}^{\mathrm{o}}%
}

\begin{document}
\title{A Variance Inequality for Meromorphic Functions under Exterior Probability}
\author{Swagatam Sen}
\date{July 1, 2020}

\maketitle
\begin{abstract}
The problem of measuring an unbounded system attribute near a singularity has been discussed. Lenses have been introduced as formal objects to study increasingly precise measurements around the singularity and a specific family of lenses called Exterior probabilities have been investigated. It has been shown that under such probabilities, measurement variance of a measurable function around a 1st order pole on a complex manifold, consists of two separable parts - one that decreases with diminishing scale of the lenses, and the other that increases. It has been discussed how this framework can lend mathematical support to ideas of non-deterministic uncertainty prevalent at a quantum scale. In fact, the aforementioned variance decomposition allows for a minimum possible variance for such a system irrespective of how close the measurements are. This inequality is structurally similar to Heisenberg uncertainty relationship if one considers energy/momentum to be a meromorphic function of a complex spacetime.
\end{abstract}

\section{Introduction}
Emergence of discrete outcome choices upon measurement of a continuously evolving intrinsic system has been a longstanding topic of interest. It is particularly of interest in the context of quantum uncertainty which appears to be fundamental to reality and is most often characterised by the continuous wave function. However, upon measurement such a continuous system collapses to discrete outcome choices. 

There has been substantial amount of work on developing a mathematical foundation for such fundamental uncertainties and evolution/collapse of continuous systems onto discrete `pure states'. Almost all of it treats measurements as operators on a Hilbert space. Born rule, put forward by Max Born in 1926 \cite{wheeler1983quantum}, proposed an interpretation of the wave function wherein the square of its amplitude represents the probability of measurement outcome in a quantum experiment. This became a core part of famous Copenhagen Interpretation of Quantum Mechanics. While it has been deemed the most satisfactory and certainly most popular interpretation for a long time, it poses a serious mathematical challenge as it remains at crossroads with classical framework of Probability Measures.

One of the most significant mathematical groundwork towards addressing this challenge was conducted by Gleason \cite{Gle}. Subsequent generalisations led to the formation of  Quantum probability via non-commutative *-algebras\cite{qp1}\cite{qp2}\cite{neumann1955mathematical}. However, the approach merely calls to focus how probabilities work differently at quantum scale, but doesn't reconcile the classical and quantum probabilities under a general framework that applies at all scales. In that sense, it harbours much of the same concerns around non-emergent nature of quantum mechanics in general. Some of the more recent works in the area\cite{belavkin2005dynamical}\cite{belavkin2005reconstruction} focuses on measurement as an `averaging' operation, albeit in the operator space.

In contrast to that, our primary concern in this paper is to investigate the possibility of a classical probability measure reproducing some of the known quantum effects. Also the approach can be described as a geometric one as opposed to the standard Operator theoretic approach. Motivationally it can be compared with the work of Fuch \cite{fuchs2010qbism} on Quantum Bayesianism. 

Particularly within the scope of this paper, we would focus on the nature of the underlying uncertainty as described by Heisenberg Uncertainty principle, and how such a relation can emerge from a classical probability measure. Heisenberg Uncertainty principle, despite multiple different representations, fundamentally is a probabilistic statement around the limits on accuracy of measurement of certain attributes at a quantum scale\cite{hsnbrg}. For example, if we allow $\Delta x$ and $\Delta \mu$ to be the standard error in measurement of position and momentum of a point in spacetime, then 
\begin{equation*}
\Delta x \Delta \mu \geq \frac{\hbar}{2}=const
\end{equation*}
A similar relation exists between other conjugate pairs e.g. time and Energy etc. 

To rigorously define these relationships, we would need to introduce the concept of `lenses'. 

\begin{defn}
Let $\mathscr{M}$ be an arbitrary manifold, and $\mathscr{L}=\{\Omega_\lambda, \Sigma_\lambda, P_\lambda\}_{\lambda\geq 0}$ be an indexed family of probability spaces on $\mathscr{M}$ such that $\Omega_\lambda \subset \Omega_{\lambda'}$ iff $\lambda < \lambda'$ and $\bigcap\limits_{\lambda} \Omega_\lambda = \{p\}$ for some $p \in \mathscr{M}$. $\mathscr{L}$ would be called a system of lenses around $p$ and each of the underlying probability spaces would be called a lens. $p$ would be called the focus of the system of lenses.
\end{defn}

Now a conjugate pair of attributes can be represented by $f : \mathscr{M} \mapsto \mathbb{C}^n$ where $\mathscr{M}$ is the manifold representing the domain for one of the pair of attributes, while $f$ denotes the relationship between the two. 

\begin{exmp}
Let $z=x+it$ be a coordinate on $\mathbb{C}$ with $x$ and $t$ representing the space and time coordinates respectively. Let $\pi : \mathbb{C} \mapsto \mathbb{C}$ be the complex momentum such that $\pi(z)=p(x,t) + iE(x,t)$. In this case $(z,\pi)$ represents a conjugate pair of attribute with $\mathscr{M}=\mathbb{C}$
\end{exmp}

 Let's choose lense $\mathscr{L}$ on $\mathscr{M}$ with scale index $\lambda$ such that $f$ is measurable $\forall \lambda$. That would mean we can safely talk about expectation $E_\lambda$ and standard deviation $\sigma_\lambda$ of $f$ measured through individual lenses.

\begin{defn}
Let $\mathscr{L}$ be a system of lenses on $\mathscr{M}$ with scale index $\lambda$ and $f:\mathscr{M} \mapsto \mathbb{C}^n$ is measurable $\forall \lambda$. Then $f$ is called Detectable if $E_\lambda(f)$ is independent of scale $\lambda$ and $\sigma_\lambda(f) < \infty, \forall \lambda$.
\end{defn}

Without loss generality for any detectable $f$, we would assume the expectation to be $0$ unless mentioned otherwise. 

For such a Detectable conjugate pair, Heisenberg-type uncertainty relation can be summarised as follows.
\begin{equation}\label{target}
\lambda \sigma_\lambda (f) \geq const
\end{equation}

If $f$ is continuous in an open neighbourhood $U$ of $p$ then $f|_U$ is bounded. In that case $\sigma_\lambda(f) \leq \sup\limits_U f|_U < \infty$ would also be universally bounded and would clearly not satisfy desired relationships as in (\ref{target}).  

That implies that for relationships of type (\ref{target}) to hold around a point $p$, we must have a singularity of $f$ at $p$. However, it is immediately clear that if we allow $P_\lambda$ to be absolutely continuous (w.r.t Lebesgue measure on $\mathbb{C}^n$), then $f$ can't be detectable as $\sigma_\lambda=\infty, \forall \lambda$. Consequently our search for a system of lenses that allow detectable functions to satisfy Heisenberg type relations, has to limit itself to certain types of probability measures that allows an open null set containing $p, \forall \lambda$. We'd refer these probabilities as Exterior Probability.

In the next few sections we would build a system of lenses on a complex manifold and study standard deviation of a detectable function $f$ under Exterior probabilities. 

\section{Exterior Probability}
We'll start with a simple disc of arbitrary length on $\mathbb{C}$, $D_\lambda=\{w \in \mathbb{C} \mid |w| \leq \lambda\}$. 

\begin{defn}

For a given open interval $I \subset [-\pi,\pi]$, we can define a Slice as $\Gamma_{\lambda}(I)=\{w \in \mathbb{C} \mid |w| \leq \lambda, arg(w) \in I\}$. Let $\gamma_{\lambda}(I)=\{w \in \mathbb{C} \mid |w| = \lambda, arg(w) \in I\}$ be the corresponding arc. \qed
\end{defn}
It's easy to check that both Slices and Arcs can be seen as distributive operators on the semi-ring of intervals.
\begin{rem}
\begin{enumerate}
\item{
$\Gamma_\lambda(I_1)\cap \Gamma_\lambda(I_2) = \Gamma_\lambda(I_1\cap I_2), \text{    } \gamma_\lambda(I_1)\cap \gamma_\lambda(I_2) = \gamma_\lambda(I_1\cap I_2)$
}
\item{
$\Gamma_\lambda(\bigcup\limits_{k=1}^\infty I_k) = \bigcup\limits_{k=1}^n \Gamma_\lambda(I_k), \text{     } \gamma_\lambda(\bigcup\limits_{k=1}^\infty I_k) = \bigcup\limits_{k=1}^n \gamma_\lambda(I_k)$
}
\item{
$\Gamma_\lambda(\bigsqcup\limits_{k=1}^\infty I_k) = \bigsqcup\limits_{k=1}^n \Gamma_\lambda(I_k), \text{     } \gamma_\lambda(\bigsqcup\limits_{k=1}^\infty I_k) = \bigsqcup\limits_{k=1}^n \gamma_\lambda(I_k)$
}
\end{enumerate}
\end{rem}
Space of all Slices of $D_\lambda$ renders a semi-ring structure. 
\begin{defn}
Let $E_\lambda= \{\Gamma_{\lambda}(I) \mid \forall \text{ interval }I \subset [-\pi,\pi] \}$ be the collection of Slices on $D_\lambda$. \qed
\end{defn}
\begin{prop}
$E_\lambda$ is a semi-ring, $\forall \lambda$ 
\end{prop}
\begin{proof}
Proof follows trivially from the fact that the space of all intervals $I$, forms a semi-ring on $[-\pi,\pi]$.\\
Let $\Gamma_\lambda(I_1)$ and $\Gamma_\lambda(I_2)$ be two Slices. It's easy to see that, $\Gamma_\lambda(I_1) \cap \Gamma_\lambda(I_2) = \Gamma_\lambda(I_1 \cap I_2) \in E_\lambda$.
Also, $\Gamma_\lambda(I_1) \setminus \Gamma_\lambda(I_2) = \Gamma_\lambda(I_1 \setminus I_2) = \Gamma_\lambda(\bigcup\limits_k C_k)= \bigcup\limits_k \Gamma_\lambda(C_k)$, where $\{C_k\}_{k}$ are a collection of intervals.
\end{proof}

Of course then we can extend it to the $\sigma$-algebra it generates.

\begin{defn}
Let $\mathscr{E}_\lambda=\sigma(E_\lambda)$ be the $\sigma$-algebra generated by $E_\lambda$. \qed
\end{defn}

Now we can start to build the measure, first on the semi-ring.
\begin{defn}
Let $\Gamma_{\lambda}(I) \in E_\lambda$. Then $\mu_\lambda : E_\lambda \mapsto \mathbb{C}$ can be defined as $\mu_\lambda(\Gamma_\lambda(I)) = \frac{1}{2\pi i}\oint\limits_{\gamma_{\lambda}(I)} \frac{1}{w}$ \qed
\end{defn}

\begin{lem}
$\mu_\lambda$ is $\sigma$-additive on $E_\lambda$
\end{lem}

\begin{proof}
Let $\Gamma = \bigsqcup\limits_{k=1}^{\infty} \Gamma_\lambda(I_k) \in E_\lambda$. That means $\bigsqcup\limits_{k=1}^\infty I_k$ is an interval. That would allow us to write $\mu_\lambda$ as,
\begin{equation*}
\begin{aligned}
\mu_\lambda(\Gamma) &= \mu_\lambda(\bigsqcup\limits_{k=1}^{\infty} \Gamma_\lambda(I_k))   = \mu_\lambda(\Gamma_\lambda(\bigsqcup\limits_{k=1}^\infty I_k))
\\ &=\frac{1}{2\pi i}\oint\limits_{\gamma_\lambda(\bigsqcup\limits_{k=1}^\infty I_k)} \frac{1}{w}  = \frac{1}{2\pi i}\oint\limits_{\bigsqcup\limits_{k=1}^{\infty} (\gamma_\lambda(I_k) )} \frac{1}{w} \\ &= \sum\limits_{k=1}^{\infty} \frac{1}{2\pi i}\oint\limits_{\gamma_\lambda(I_k)} \frac{1}{w}  = \sum\limits_{k=1}^{\infty} \mu_\lambda(\Gamma_\lambda(I_k))
\end{aligned}
\end{equation*}
\end{proof}

\begin{cor}
$\mu_\lambda$ can be uniquely extended as a complex probability on $\mathscr{E}_\lambda$. \qed
\end{cor}

Let $(\mathscr{E}^n_\lambda, \mu^n_\lambda)$ denote the product probability space on $D^n_\lambda$ where $\mu^n_\lambda(\varprod \Gamma_k) = \prod\limits_k \mu_\lambda(\Gamma_k)$ for $\Gamma_k \in \mathscr{E}_\lambda, \forall k$. We'd refer to $\mu^n_\lambda$ as the Exterior Probability on the poly-disc $D^n_\lambda$. We'd define a Unit Lens as the filtration of probability spaces $\mathscr{L} = \{(D^n_\lambda, \mathscr{E}^n_\lambda, \mu^n_\lambda)\}_{\lambda \downarrow 0}$.

\begin{defn}
A Convergent Lens on $\mathbb{C}^n$ around $0$, is a filtration of probability spaces $\{D^n_\lambda, \mathscr{E}^n_\lambda, \tau_\lambda)\}_{\lambda \downarrow 0}$ such that $\tau_\lambda$ is absolutely continuous with respect to $\mu^n_\lambda$, $\forall \lambda$. \qed
\end{defn}

For the rest of this paper we'd focus on Unit Lenses. However, similar results can be recovered for a more general convergent lenses.

\section{Variance Inequality for Meromorphic Functions}
Let $\mathscr{F}$ be the vector space of all meromorphic function $f : \mathbb{C}^n \mapsto \mathbb{C}^k$ which has a potential pole of order at most 1 at $0$ and is $\mu^n_\lambda$-measurable over $D^n_\lambda$. 

\begin{rem}
$\mu^n_\lambda$ induces an Expectation $E(\cdot)$ and an inner product $\langle \cdot , \cdot \rangle$ on $\mathscr{F}$ as 
\begin{itemize}
\item{$E(f)=\int\limits_{D_\lambda} f d\mu^n_\lambda=\frac{1}{(2\pi i)^n}\oint\limits_{D_\lambda}\frac{f(w)}{\prod\limits_{\alpha}w_\alpha}$,}
\item{$\langle f,g \rangle = \int\limits_{D_\lambda} f^*g d\mu^n_\lambda= \frac{1}{(2\pi i)^n}\oint\limits_{D_\lambda}\frac{f(w)^*g(w)}{\prod\limits_{\alpha}w_\alpha} $}
\end{itemize}
\end{rem}

\begin{rem}
We can write $f$ as a sum of its core, principal and analytical components. 
\begin{itemize}
\item{$f_0^\alpha = \frac{1}{(2\pi i)^n}\oint\limits_{D_\lambda}{\frac{f^\alpha(\mathbf{w})}{\prod_{\mu=1}^{3} w^\mu}} $ is the core of $f$.}
\item{$f_P^\alpha =  \sum_{\beta=1}^{n} \frac{\eta^{\alpha}_\beta}{z_\beta}$ is the principal component of $f$ where $\eta^\alpha_\beta = \frac{1}{(2\pi i)^n}\oint\limits_{D_\lambda}{\frac{f^\alpha(\mathbf{w})}{\prod_{\mu\ne\beta} w_\mu}}$ is the matrix of residues. }
\item{$f_A = f - f_0 - f_P$ is the Analytic component of $f$ and can be expressed as a power series within a radius of convergence $R$ i.e. $f_A^\alpha =  \frac{1}{(2\pi i)^n}\sum_{\beta=1}^{n} [\oint\limits_{D_\lambda}{\frac{f^\alpha(\mathbf{w})}{w_\beta\prod_{\mu} w_\mu}} ] z^\beta + P^\alpha(z) = \sum_{\beta=1}^{n} \mathscr{D}^{\alpha\beta}z_\beta + P^\alpha(z)$,  $\forall z$ such that $|z^\beta| < R$  $\forall \beta$ where $\mathscr{D}^{\alpha\beta}=\frac{\partial f^\alpha}{\partial w_\beta}|_{w=0}$ and $P^\alpha$ is a power series of degree at least 2}
\end{itemize}
\end{rem}

\begin{rem}
For $f \in \mathscr{F}$, let $\eta$ be the residue matrix and $f_0$ be the core. Then
\begin{enumerate}
\item{$E(f)=f_0$ is independent of $\lambda$}
\item{$\langle \overline{z},f \rangle = \lambda^2 \langle \frac{1}{z},f \rangle = Tr(\eta)$}
\item{$\langle z,f \rangle = \lambda^2 \langle \frac{1}{\overline{z}},f \rangle = Tr(\mathscr{D})$}
\end{enumerate}
\end{rem}

Additionally we can characterise the higher order terms in $f_A$

\begin{prop} \label{prop1}
Let $f \in \mathscr{F}$ and let $f_A$ be the analytic component of $f$. Also let $f_A^\alpha(\mathbf{w})= \mathscr{D}^{\alpha\beta}w_\beta + P^\alpha(\mathbf{w}), \forall \alpha$ where $P^\alpha$ is a power series of degree at least 2. Then the following statements are true
\begin{enumerate}
\item{$\oint\limits_{D_\lambda} \frac{P^\alpha(\mathbf{w})}{\prod\limits_{\gamma}w_\gamma} = \oint\limits_{D_\lambda} \frac{\overline{P^\alpha(\mathbf{w})}}{\prod\limits_{\gamma}w_\gamma} = 0 \text{       }\forall \alpha.$}
\item{$\oint\limits_{D_\lambda} \frac{P^\alpha(\mathbf{w})}{\prod\limits_{\gamma \ne \delta}w_\gamma} = \oint\limits_{D_\lambda} \frac{\overline{P^\alpha(\mathbf{w})}}{\prod\limits_{\gamma \ne \delta}w_\gamma} = 0 \text{            }\forall \alpha, \delta.$}
\item{$\oint\limits_{D_\lambda} \frac{P^\alpha(\mathbf{w})}{w_\delta\prod\limits_{\gamma \ne \delta}w_\gamma} = \oint\limits_{D_\lambda} \frac{\overline{P^\alpha(\mathbf{w})}}{w_\delta \prod\limits_{\gamma}w_\gamma} = 0 \text{            }\forall \alpha, \delta.$}
\item{$\oint\limits_{D_\lambda} \frac{(\overline{P^{\alpha}(\mathbf{w})})P^{\alpha}(\mathbf{w})}{\prod\limits_{\gamma}w_\gamma} = \sum\limits_{i,j} \oint\limits_{D_\lambda} \frac{(\overline{P^{\alpha}_i(\mathbf{w})})P^{\alpha}_j(\mathbf{w})}{\prod\limits_{\gamma}w_\gamma}  = 0 \text{             }\forall \alpha$}
\end{enumerate}
\end{prop}

\begin{proof}
Let $D^{(\beta)}_\lambda = \{|w_1|=\lambda\}\otimes ... \otimes \{|w_{\beta-1}|=\lambda\}\otimes\{|w_{\beta+1}|=\lambda\}\otimes ... \otimes \{|w_n|=\lambda\}$.\\
\\
If $f^\alpha_A(\mathbf{w})= \mathscr{D}_{\alpha\beta}w_\beta + P^\alpha(\mathbf{w})$   $\forall \alpha$ where $P^\alpha$ is a power series of lowest degree at least 2, convergent within the radius $R$, that means $P^\alpha(\mathbf{w}) = \sum P^{\alpha}_i(\mathbf{w})$ where $\forall \alpha, \forall i$ either of these two possibilities are true-
\begin{itemize}
\item{Case I - $\exists \beta_i$ such that $P^{\alpha}_i(\mathbf{w})=w_{\beta_i}^{r_i}Q^{\alpha}_i(\mathbf{w}^{(\beta_i)})$ with $r_i \geq 2$ where $Q^{\alpha}_i$ is a convergent power series on $\mathbf{w}^{(\beta_i)}=(w_1,...,w_{\beta_i-1},w_{\beta_i+1},...,w_n)$.}
\item{Case II - $\exists \beta_{i|1} < \beta_{i|2}$ such that $P^{\alpha}_i(\mathbf{w})=w_{\beta_i|1}w_{\beta_{i|2}}Q^{\alpha}_i(\mathbf{w}^{(\beta_{i|1},\beta_{i|2})})$ where $Q^{\alpha}_i$ is a convergent power series on $\mathbf{w}^{(\beta_{i|1},\beta_{i|2})}=(w_1,...,w_{\beta_i-1},w_{\beta_i+1},...,w_{\beta_{i|2}-1},w_{\beta_{i|2}+1},...,w_n)$.}
\end{itemize}

\begin{enumerate}
\item{
\begin{enumerate}
\item{In that case,
\begin{equation*}
\oint\limits_{D^n_\lambda} \frac{P^{\alpha}(\mathbf{w})}{\prod\limits_{\gamma}w_\gamma} = 0
\end{equation*}
because $P^{\alpha}$ is analytic at $0$.
}
\item{
Also
\begin{equation*}
\oint\limits_{D^n_\lambda} \frac{\overline{P^{\alpha}_i(\mathbf{w})}}{\prod\limits_{\gamma}w_\gamma} = \lambda^{2r_i}\oint\limits_{D^{(\beta_i)}_\lambda} \frac{\overline{Q^{\alpha}_i(\mathbf{w}^{(\beta_i)})}}{\prod\limits_{\gamma \ne \beta_i}w_\gamma} \oint\limits_{|w_{\beta_i}|=\lambda} \frac{1}{w_{\beta_i}^{r_i+1}} = 0 \text{       } \forall \alpha, i.
\end{equation*}
\\
\\
Hence it follows that 
\begin{equation*}
\label{poly1}
\oint\limits_{D^n_\lambda} \frac{\overline{P^\alpha(w)}}{\prod\limits_{\gamma}w_\gamma} = 0 \text{     }\forall \alpha.
\end{equation*}
}
\end{enumerate}
}
\item{

\begin{enumerate}
\item{Working with the power series components $P^{\alpha}_i$, we know that,
\begin{equation*}
\oint\limits_{D^n_\lambda} \frac{P^{\alpha}_i(\mathbf{w})}{\prod\limits_{\gamma \ne \delta}w_\gamma} = \oint\limits_{D^{(\beta_i)}_\lambda} \frac{Q^{\alpha}_i(\mathbf{w}^{(\beta_i)})}{\prod\limits_{\gamma \ne \beta_i,\delta}w_\gamma}\oint\limits_{|w_{\beta_i}|=\lambda} w_{\beta_i}^{r_i-s_i} = 0
\end{equation*}
where $s_i= I(\beta_i \ne \delta)$ and $r_i\geq 1$
}
\item{Same would apply for the conjugate with a slightly different calculation.
\begin{equation*}
\oint\limits_{D_\lambda} \frac{\overline{P^{\alpha}_i(\mathbf{w})}}{\prod\limits_{\gamma \ne \delta}w_\gamma} = \lambda^{2r_i}\oint\limits_{D_\lambda} \frac{\overline{Q^{\alpha}_i(\mathbf{w}^{(\beta_i)})}}{w_{\beta_i}^{r_i+s_i}\prod\limits_{\gamma \ne \beta_i,\delta}w_\gamma}
\end{equation*}
where $s_i = I(\beta_i \ne \delta)$

Unless $\beta_i=\delta$ and $r_i=1$, we could simply write this,
\begin{equation*}
\oint\limits_{D_\lambda} \frac{\overline{P^{\alpha}_i(\mathbf{w})}}{\prod\limits_{\gamma \ne \delta}w_\gamma} = \lambda^{2r_i}\oint\limits_{D^{(\beta_i)}_\lambda} \frac{\overline{Q^{\alpha}_i(\mathbf{w}^{(\beta_i)})}}{\prod\limits_{\gamma \ne \beta_i,\delta}w_\gamma}\oint\limits_{|w_{\beta_i}|=\lambda}\frac{1}{w_{\beta_i}^{r_i+s_i}} = 0
\end{equation*}
as $r_i+s_i>1$

For the case when $\beta_i=\delta$ and $r_i=1$, we could invoke Case II and assume without loss of generality that $\beta_{i|1} \ne \delta$. That would allow us to write,
\begin{equation*}
\oint\limits_{D_\lambda} \frac{\overline{P^{\alpha}_i(\mathbf{w})}}{\prod\limits_{\gamma \ne \delta}w_\gamma} = \lambda^4\oint\limits_{D^{(\beta_{i|1})}_\lambda} \frac{\overline{Q^{\alpha}_i(\mathbf{w}^{(\beta_{i|1},\beta_{i|2})})}}{w_{\beta_{i|2}}\prod\limits_{\gamma \ne \beta_{i|1},\delta}w_\gamma} \oint\limits_{|w_{\beta_{i|1}}|=\lambda} \frac{1}{w^2_{\beta_{i|1}}}= 0
\end{equation*}

This of course leads to 
\begin{equation*}
\oint\limits_{D_\lambda} \frac{\overline{P^{\alpha}(\mathbf{w})}}{\prod\limits_{\gamma \ne \delta}w_\gamma} = \sum\limits_{i}\oint\limits_{D_\lambda} \frac{\overline{P^{\alpha}_i(\mathbf{w})}}{\prod\limits_{\gamma \ne \delta}w_\gamma} = 0 \text{         } \forall \alpha,\delta
\end{equation*}

}
\item{
For the other identities,

\begin{equation*}
\label{poly2}
\oint\limits_{D_\lambda} \frac{P^{\alpha}_i(\mathbf{w})}{w_\delta \prod\limits_{\gamma}w_\gamma} =  \oint\limits_{D^{(\beta_i)}_\lambda} \frac{Q^{\alpha}_i(\mathbf{w}^{(\beta_i)})}{w_\delta^{s_i} \prod\limits_{\gamma \ne \beta_i}w_\gamma} \oint\limits_{|w_{\beta_i}|=\lambda} w_{\beta_i}^{r_i+s_i-1} = 0 \text{           } \forall \alpha,\delta.
\end{equation*}
where $s_i= I(\beta_i \ne \delta)$

That implies,
\begin{equation*}
\oint\limits_{D_\lambda} \frac{P^{\alpha}(\mathbf{w})}{w_\delta\prod\limits_{\gamma}w_\gamma} = \sum\limits_{i}\oint\limits_{D_\lambda} \frac{P^{\alpha}_i(\mathbf{w})}{w_\delta\prod\limits_{\gamma}w_\gamma} = 0 \text{         } \forall \alpha,\delta
\end{equation*}
}
\item{
By a similar argument,
\begin{equation*}
\oint\limits_{D_\lambda} \frac{\overline{P^{\alpha}_i(\mathbf{w})}}{w_\delta \prod\limits_{\gamma}w_\gamma} =  \lambda^{2r_i}\oint\limits_{D^{(\beta_i)}_\lambda} \frac{\overline{Q^{\alpha}_i(\mathbf{w}^{(\beta_i)})}}{w_\delta^s \prod\limits_{\gamma \ne \beta_i}w_\gamma} \oint\limits_{|w_{\beta_i}|=\lambda} \frac{1}{w_{\beta_i}^{r_i-s_i+1}} = 0 \text{           } \forall \alpha,\delta.
\end{equation*}
}
\end{enumerate}

}
\item{
Also it's easy to check that,
\begin{equation*}
\begin{aligned}
\oint\limits_{D_\lambda} \frac{(\overline{P^{\alpha}_i(\mathbf{w})})P^{\alpha}_j(\mathbf{w})}{\prod\limits_{\gamma}w_\gamma} = \lambda^{r_i+r_j}\oint\limits_{D^{(\beta_i,\beta_j)}_\lambda} \frac{\overline{Q^{\alpha}_i(w^{(\beta_i)})}Q^{\alpha}_j(w^{(\beta_j)})}{\prod\limits_{\gamma \ne \beta_i,\beta_j}w_\gamma} &\oint\limits_{|w_{\beta_i}|=\lambda} \frac{1}{w_{\beta_i}^{r_i+1}} \oint\limits_{|w_{\beta_j}|=\lambda} w_{\beta_j}^{r_j-1}\\
&= 0, \text{      }\forall \alpha, i
\end{aligned}
\end{equation*}
\\
\\
which means that,
\begin{equation*}
\label{poly3}
\oint\limits_{D_\lambda} \frac{(\overline{P^{\alpha}(\mathbf{w})})P^{\alpha}(\mathbf{w})}{\prod\limits_{\gamma}w_\gamma} = \sum\limits_{i,j} \oint\limits_{D_\lambda} \frac{(\overline{P^{\alpha}_i(\mathbf{w})})P^{\alpha}_j(\mathbf{w})}{\prod\limits_{\gamma}w_\gamma}  = 0 \text{         } \forall \alpha.
\end{equation*}
}
\end{enumerate}
\end{proof}

\begin{lem}\label{lem1}
Let $f \in \mathscr{F}$ and let $f_0, f_P$ and $f_A$ be the core, principal and analytic components of $f$. If $D_\lambda$ is equipped with an exterior probability $\nu_\lambda$, then it has following properties
\begin{enumerate}
\item{$E(f_P)=E(f_A)=E(\overline{f_P})=E(\overline{f_A})=0$}
\item{$\langle f_A, f_P \rangle = \langle f_P, f_A \rangle = 0$}
\item{$\|f_P\|^2 = \langle f_p, f_p \rangle = \frac{Tr(\eta)}{\lambda^2}$}
\item{$\|f_A\|^2 = \lambda^2Tr(\mathscr{D})$} 
\end{enumerate}
\end{lem}
\begin{proof}

We can look to prove the desired results utilising Proposition \ref{prop1}.
\begin{enumerate}
\item{
\begin{enumerate}
\item{
We'll start by showing the principal component has vanishing expectation.
\begin{equation*}
\begin{aligned}
E(f_P^\alpha) = \frac{1}{(2\pi i)^n}\sum\limits_{\beta}\oint\limits_{D_\lambda}\frac{\eta_\beta^\alpha}{w_\beta\prod\limits_{\gamma}w_\gamma} = \frac{1}{(2\pi i)^n}\sum\limits_{\beta}\eta_\beta^\alpha \oint\limits_{D^{(\beta)}_\lambda}\frac{1}{\prod\limits_{\gamma \ne \beta}w_\gamma}\oint\limits_{|w_\beta|=\lambda}\frac{1}{w_\beta^2} =0,         \forall \alpha
\end{aligned}
\end{equation*}
}
\item{
and the same is true for its conjugate.
\begin{equation*}
\begin{aligned}
E(\overline{f_P^\alpha}) = \frac{1}{(2\pi i)^n}\sum\limits_{\beta}\oint\limits_{D_\lambda}\frac{\overline{\eta_\beta^\alpha}}{\overline{w_\beta}\prod\limits_{\gamma}w_\gamma} = \frac{1}{(2\pi i)^n}\frac{1}{\lambda^2}\sum\limits_{\beta}\overline{\eta_\beta^\alpha} \oint\limits_{D^{(\beta)}_\lambda}\frac{1}{\prod\limits_{\gamma \ne \beta}w_\gamma}\oint\limits_{|w_\beta|=\lambda}1=0,      \forall \alpha
\end{aligned}
\end{equation*}
}
\item{
Similar calculation can be done for the analytic component as well, with identical outcome.
\begin{equation*}
\begin{aligned}
E(f_A) = \frac{1}{(2\pi i)^n}\sum\limits_{\beta}\oint\limits_{D_\lambda}\frac{f_A(\mathbf{w})}{\prod\limits_{\gamma}w_\gamma} =f_A(0) = 0
\end{aligned}
\end{equation*}
 as $f_A$ is analytic at 0.
 }
 \item{
 And the same holds for its conjugate again.
 
 \begin{equation*}
 \begin{aligned}
 E(\overline{f_A^\alpha}) = \frac{1}{(2\pi i)^n}\sum\limits_{\beta}\oint\limits_{D_\lambda}\frac{\overline{\mathscr{D}^{\alpha\beta}w_\beta + P^\alpha(\mathbf{w})}}{\prod\limits_{\gamma}w_\gamma} 
&= \frac{\lambda^2}{(2\pi i)^n}\sum\limits_{\beta} \overline{\mathscr{D}^{\alpha\beta}} \oint\limits_{D^{(\beta)}_\lambda}\frac{1}{\prod\limits_{\gamma \ne \beta}w_\gamma}\oint\limits_{|w_\beta|=\lambda}\frac{1}{w_\beta^2} \\
&+ \frac{1}{(2\pi i)^n} \oint\limits_{D_\lambda} \frac{\overline{P^\alpha(\mathbf{w})}}{\prod\limits_{\gamma}w_\gamma}=0, \forall \alpha
\end{aligned}
\end{equation*}
}
\end{enumerate}
}
\item{
\begin{enumerate}
\item{
Now we would turn our attention to the inner product between analytic and principal components.
\begin{equation*}
\begin{aligned}
\langle f_A, f_P \rangle &= \frac{1}{(2\pi i)^n}\sum\limits_{\beta}\sum\limits_{\delta}\sum\limits_\alpha\oint\limits_{D_\lambda}\frac{(\overline{\mathscr{D}^{\alpha\beta}w_\beta + P^\alpha(\mathbf{w})})(\eta_\delta^\alpha/w_\delta)}{\prod\limits_{\gamma}w_\gamma} \\
&=\frac{\lambda^2}{(2\pi i)^n}\sum\limits_{\beta} \sum\limits_{\delta}\sum\limits_\alpha\overline{\mathscr{D}^{\alpha\beta}} \eta_\delta^\alpha \oint\limits_{D^{(\beta)}_\lambda}\frac{1}{w_\delta^{s}\prod\limits_{\gamma \ne \beta}w_\gamma}\oint\limits_{|w_\beta|=\lambda}\frac{1}{w_\beta^{3-s}} \\
&+ \frac{1}{(2\pi i)^n} \sum\limits_{\delta}\sum\limits_\alpha\eta_\delta^\alpha \oint\limits_{D_\lambda} \frac{ \overline{P^\alpha(\mathbf{w})}}{w_\delta \prod\limits_{\gamma}w_\gamma}=0
\end{aligned}
\end{equation*}
where $s = I(\beta \ne \delta)$
}
\item{
Change of order, doesn't make a difference to the outcome, of course.
\begin{equation*}
\begin{aligned}
\langle f_P, f_A \rangle &= \frac{1}{(2\pi i)^n}\sum\limits_{\beta}\sum\limits_{\delta}\oint\limits_{D_\lambda}\frac{(\mathscr{D}^{\alpha\beta}w_\beta + P^\alpha(\mathbf{w}))(\overline{\eta_\delta^\alpha/w_\delta})}{\prod\limits_{\gamma}w_\gamma} \\
&=\frac{1}{\lambda^2}\frac{1}{(2\pi i)^n}\sum\limits_{\beta} \sum\limits_{\delta}\sum\limits_\alpha\mathscr{D}^{\alpha\beta} \overline{\eta_\delta^\alpha} \oint\limits_{D^{(\beta)}_\lambda}\frac{w_\delta^{s}}{\prod\limits_{\gamma \ne \beta}w_\gamma}\oint\limits_{|w_\beta|=\lambda}w_\beta^{s+1} \\
&+ \frac{1}{\lambda^2}\frac{1}{(2\pi i)^n} \sum\limits_{\delta}\sum\limits_\alpha\overline{\eta_\delta^\alpha} \oint\limits_{D_\lambda} \frac{ P^\alpha(\mathbf{w})}{\prod\limits_{\gamma \ne \delta}w_\gamma}=0
\end{aligned}
\end{equation*}
where $s = I(\beta \ne \delta)$
}
\end{enumerate}
}
\item{
Finally we'll concentrate on the self interaction terms. 

\begin{enumerate}
\item{First for the principal component.
\begin{equation*}
\begin{aligned}
\|f_P\|^2 &= \langle f_P, f_P \rangle = \frac{1}{(2\pi i)^n}\sum\limits_{\beta}\sum\limits_{\delta}\sum\limits_\alpha \oint\limits_{D_\lambda}\frac{(\overline{\eta_\beta^\alpha/w_\beta})(\eta_\delta^\alpha/w_\delta)}{\prod\limits_{\gamma}w_\gamma} \\
&= \frac{1}{\lambda^2}\frac{1}{(2\pi i)^n}\sum\limits_{\beta}\sum\limits_{\delta \ne \beta}\sum\limits_\alpha \overline{\eta_\beta^\alpha}\eta_\delta^\alpha \oint\limits_{D^{|(\beta)}_\lambda} \frac{1}{w_\delta\prod\limits_{\gamma}w_\gamma} \oint\limits_{|w_\beta|=\lambda} w_\beta \\
&+\frac{1}{\lambda^2}\frac{1}{(2\pi i)^n}\sum\limits_{\beta}\sum\limits_\alpha \overline{\eta_\beta^\alpha}\eta_\beta^\alpha \oint\limits_{D_\lambda} \frac{1}{\prod\limits_{\gamma}w_\gamma} \\
&= 0 + \frac{1}{\lambda^2}\sum\limits_{\beta}\sum\limits_\alpha \overline{\eta_\beta^\alpha}\eta_\beta^\alpha = \frac{Tr(\eta^*\eta)}{\lambda^2}
\end{aligned}
\end{equation*}
}
\item{
And then for the analytic one.

\begin{equation*}
\begin{aligned}
\langle f_A, f_A \rangle &= \frac{1}{(2\pi i)^n}\sum\limits_{\alpha}\oint\limits_{D_\lambda}\frac{(\sum\limits_{\beta}\mathscr{D}^{\alpha\beta}w_\beta + P^\alpha(\mathbf{w}))\overline{(\sum\limits_{\delta}\mathscr{D}^{\alpha\delta}w_\delta + P^\alpha(\mathbf{w}))}}{\prod\limits_{\gamma}w_\gamma} \\
&=\frac{\lambda^2}{(2\pi i)^n}\sum\limits_{\beta} \sum\limits_{\delta \ne \beta}\sum\limits_\alpha\mathscr{D}^{\alpha\beta} \overline{\mathscr{D}^{\alpha\delta}} \oint\limits_{D^{(\beta)}_\lambda}\frac{1}{w_\delta\prod\limits_{\gamma}w_\gamma}\oint\limits_{|w_\beta = \lambda|} w_\beta \\
&+ \frac{\lambda^2}{(2\pi i)^n}\sum\limits_{\beta} \sum\limits_\alpha\mathscr{D}^{\alpha\beta} \overline{\mathscr{D}^{\alpha\beta}} \oint\limits_{D_\lambda}\frac{1}{\prod\limits_{\gamma}w_\gamma} +\frac{\lambda^2}{(2\pi i)^n} \sum\limits_{\delta}\sum\limits_\alpha\overline{\mathscr{D}^{\alpha\delta}} \oint\limits_{D^\lambda} \frac{ P^\alpha(\mathbf{w})}{w_\delta\prod\limits_{\gamma }w_\gamma} \\
&+\frac{1}{(2\pi i)^n} \sum\limits_{\delta}\sum\limits_\alpha\mathscr{D}^{\alpha\beta}\oint\limits_{D_\lambda} \frac{ \overline{P^\alpha(\mathbf{w})}}{\prod\limits_{\gamma \ne \delta}w_\gamma} 
+ \frac{1}{(2\pi i)^n} \sum\limits_{\alpha}\oint\limits_{D_\lambda} \frac{(\overline{P^{\alpha}(w)})P^{\alpha}(w)}{\prod\limits_{\gamma}w_\gamma} \\
&= \lambda^2\sum\limits_{\beta} \sum\limits_\alpha\mathscr{D}^{\alpha\beta} \overline{\mathscr{D}^{\alpha\beta}} = \lambda^2 Tr(\mathscr{D}^{*}\mathscr{D})
\end{aligned}
\end{equation*}
}
\end{enumerate}
}
\end{enumerate}
\end{proof}  

With these results, we can now embark on proving the desired relation between the variance of a meromorphic function over a poly-disc around its singular pole.
\begin{thm}\label{thm1}
For a meromorphic function $f \in \mathscr{F}$ such that there is a single pole at $0$ of order at most 1, the variance of $f$ over the Unit Lens $\mathscr{L}$ using exterior probability $\mu^n_\lambda$ is given by, 
\begin{equation*}
V_\lambda(f)= \langle f,f \rangle - \|E(f)\|^2 =  \frac{Tr(\eta^*\eta)}{\lambda^2} + \lambda^2 Tr(\mathscr{D}^{*}\mathscr{D})
\end{equation*}
\end{thm}
\begin{proof}
Self interaction of $f$ can be written as 
\begin{equation*}
\begin{aligned}
\langle f, f \rangle &= \langle f_0,f_0 \rangle + \langle f_P, f_P \rangle + \langle f_A,f_A \rangle + \langle f_0,f_P \rangle + \langle f_P,f_0 \rangle + \langle f_0, f_A \rangle + \langle f_A, f_0 \rangle + \langle f_P, f_A \rangle + \langle f_A, f_P \rangle \\
&= \|E(f)\|^2 + \|f_P\|^2 + \|f_A\|^2 + f_0^*E(f_P) + E(f_P)^*f_0 + f_0^*E(f_A) + E(f_A)^*f_0 + \langle f_P, f_A \rangle + \langle f_A, f_P \rangle \\
&= \|E(f)\|^2 + \frac{Tr(\eta^*\eta)}{\lambda^2} + \lambda^2 Tr(\mathscr{D}^{*}\mathscr{D})
\end{aligned}
\end{equation*}
This would imply
\begin{equation*}
V_\lambda(f)= \frac{Tr(\eta^*\eta)}{\lambda^2} + \lambda^2 Tr(\mathscr{D}^{*}\mathscr{D})
\end{equation*}
\end{proof}

\begin{cor}
If $\sigma_\lambda$ is the standard error of $f$ under the exterior probability, then $\lambda \sigma_\lambda \geq \sqrt{Tr(\eta^*\eta)}$
\end{cor}

\section{Generalisation on a Complex manifold}
Let $K=\big(\mathscr{M},\mathscr{A}\big)$ be a complex manifold of dimension $n$ where $\mathscr{A}$ is a holomorphic structure. 

Let $D_\lambda = \{w \in \mathbb{C}^n \mid |w_\alpha| \leq \lambda,  \forall \alpha\}$ be a closed poly-disc in $\mathbb{C}^n$. 

\begin{defn}
For $p_0 \in \mathscr{M}$ and $\lambda>0$, Let $S=S_{\lambda}=z^{-1}(D_\lambda)$ for some $(z,U) \in \mathscr{A}$ and $\mathscr{Z}(S_\lambda) =\{z' = g\circ z\mid \text{g is holomorphism on }z(U)\}$. We'd refer to $\big(S_\lambda, \mathscr{Z}(S_\lambda)\big)$ as a poly-disc on $\mathscr{M}$ around $p_0$ \qed
\end{defn}

\begin{rem}
For a poly-disc $S_\lambda$ under a morph $z$, let $\mathscr{S}_\lambda=z^{-1}(\mathscr{E}^n_\lambda)$. Since $\mathscr{E}^n_\lambda$ is a $\sigma$-algebra and $z$ is bijective, we can conclude that $\mathscr{S}_\lambda$ is a $\sigma$-algebra $\forall z \in \mathscr{Z}(S_\lambda)$.
\end{rem}

\begin{rem}
$(S_\lambda, \mathscr{S}_\lambda)$ is a measurable space for every morph, $z \in \mathscr{Z}(S_\lambda)$. 
\end{rem}

\begin{prop}
If $\nu_\lambda : \mathscr{S}_\lambda \mapsto \mathbb{C}$ is defined as  $\nu_\lambda(A)=\mu^n_\lambda(z(A)), \forall A \in \mathscr{S}_\lambda, \text{ under a morph  }z $, then $\nu_\lambda$ is $\sigma$-additive and hence a complex measure on  $(S_\lambda, \mathscr{S}_\lambda)$
\end{prop}  
\begin{proof}
Let $\{A_i\}_{i=1}^\infty$ be.a sequence of pairwise disjoint sets in $\mathscr{S}_\lambda$. Since the underlying morph $z$ is bijective, that would mean that $\{z(A_i)\}_{i=1}^\infty$ are also pairwise disjoint. Also, $z(\bigsqcup\limits_{i=1}^\infty A_i) = \bigsqcup\limits_{i=1}^\infty z(A_i)$. That would naturally yield
\begin{equation*}
\nu_\lambda(\bigsqcup\limits_{i=1}^\infty A_i)= \mu_\lambda(z(\bigsqcup\limits_{i=1}^\infty A_i)) = \mu_\lambda(\bigsqcup\limits_{i=1}^\infty z(A_i)) = \sum\limits_{i=1}^\infty \mu_\lambda(z(A_i)) = \sum\limits_{i=1}^\infty \nu_\lambda(A_i)
\end{equation*}
\end{proof}

\begin{defn}
For a manifold $\mathscr{M}$ and $p_0 \in \mathscr{M}$, An Unit Lens is the Filtration of probability spaces defined by $\mathscr{L}=\{(S_\lambda, \mathscr{S}_\lambda,\nu_\lambda)\}_{\lambda \downarrow 0}$ indexed by a scale parameter $\lambda$.
\end{defn}

Let's fix a poly-disc $S_\lambda$ and a particular morph $z$. Let $p_0 = z^{-1}(0) \in S_\lambda$.
Let $\mathscr{P}$ be the vector space of all meromorphic function $\psi : \mathscr{M} \mapsto \mathbb{C}^k$ which has a pole of at most order 1 at $p_0$ i.e $\psi_z=\psi \circ z^{-1} : \mathbb{C}^n \mapsto \mathbb{C}^k$ is analytic everywhere except for a (potential) pole at $z(p_0)=0$. 

\begin{rem}
$\nu_\lambda$ induces an Expectation $E(\cdot)$ and an inner product $\langle \cdot , \cdot \rangle$ on $\mathscr{P}$ as 
\begin{itemize}
\item{$E(\psi)=\int\limits_{S_\lambda} \psi d\nu_\lambda=\int\limits_{D_\lambda} \psi_z d\mu^n_\lambda$,}
\item{$\langle \psi,\phi \rangle = \int\limits_{S_\lambda} \psi^*\phi d\nu_\lambda= \int\limits_{D_\lambda} \psi^*_z \phi_z d\mu^n_\lambda$}
\end{itemize}
\end{rem}

\begin{rem}
We can write $\psi_z = \psi \circ z^{-1}$ as a sum of its core($\psi_0$), principal($\psi_P$) and analytic($\psi_A$) components. Let $\eta_z$ denote the corresponding matrix of residues and $\mathscr{D}_z$ be the Jacobian of the analytic component at $z=0$.
\end{rem}

\begin{lem}
For any given $\psi \in \mathscr{P}$, if $\eta$ and $\eta'$ denote the residue matrix under two different morphs $z, z' \in \mathscr{Z}(S_\lambda)$ respectively, then $\forall \alpha$, 
\begin{equation*}
\eta^\alpha_\beta = \eta'^\alpha_\gamma \frac{\partial z_\beta}{\partial z'_\gamma}
\end{equation*}
In other words, $\eta^\alpha_\beta$ transforms contra-variantly on $\beta$.
\end{lem}
\begin{proof}
Since both $z$ and $z'$ are morphs for $S_\lambda$, we can define $g = z'\circ z^{-1} :D_\lambda \mapsto D_\lambda$, holomorphic with $g(0)=0$. 

Let $H^\gamma_\beta = \frac{\partial z_\beta}{\partial z'_\gamma} = (g^{-1})'(0)$ be the Jacobian matrix for the coordinate transform at $z=0$. So we can write, $dw_\beta = H^\gamma_\beta dg_\gamma$.

Now,
 \begin{equation*}
 \begin{aligned}
\eta^\alpha_\beta &= \frac{1}{(2\pi i)^n}\oint\limits_{D_\lambda}\frac{\psi_z(\mathbf{w})}{\prod\limits_{\mu \ne \beta} w_\mu}\prod\limits_\mu dw_\mu 
                 = \frac{1}{(2\pi i)^n}\oint\limits_{D_\lambda}\frac{\psi_z'(g(\mathbf{w}))}{\prod\limits_{\mu \ne \beta} w_\mu}\prod\limits_\mu dw_\mu \\
                 &= \frac{1}{(2\pi i)^n}\oint\limits_{D^{(\beta)}_\lambda}\frac{1}{\prod\limits_{\mu \ne \beta} w_\mu}\prod\limits_{\mu \ne \beta} dw_\mu \oint\limits_{|w_\beta|=\lambda} 
                 (\psi'_0 + \mathscr{D}'^{\alpha\gamma}g_\gamma + P^\alpha(g(\mathbf{w}))+\frac{\eta'^\alpha_\gamma}{g_\gamma}) dw_\beta \\
                 &= \frac{1}{(2\pi i)^n}\oint\limits_{D^{(\beta)}_\lambda}\frac{1}{\prod\limits_{\mu \ne \beta} w_\mu}\prod\limits_{\mu \ne \beta} dw_\mu \oint\limits_{|w_\beta|=\lambda} 
                 \frac{\eta'^\alpha_\gamma}{g_\gamma} dw_\beta \\
                 &= \frac{1}{(2\pi i)^n}\oint\limits_{D^{(\beta)}_\lambda}\frac{1}{\prod\limits_{\mu \ne \beta} w_\mu}\prod\limits_{\mu \ne \beta} dw_\mu \oint\limits_{|w_\beta|=\lambda} 
                 \frac{\eta'^\alpha_\gamma}{g_\gamma} H^\gamma_\beta dg_\gamma \\
                 &=\eta'^\alpha_\gamma H^\gamma_\beta = \eta'^\alpha_\gamma \frac{\partial z_\beta}{\partial z'_\gamma}
 \end{aligned}
 \end{equation*}
\end{proof}

\begin{lem}
For any given $\psi \in \mathscr{P}$, if $\mathscr{D}$ and $\mathscr{D}'$ denote the Jacobian matrix under two different morphs $z, z' \in \mathscr{Z}(S_\lambda)$ respectively, then $\forall \alpha$, 
\begin{equation*}
\mathscr{D}^{\alpha\beta} = \mathscr{D}'^{\alpha\gamma} \frac{\partial z'_\gamma}{\partial z_\beta}
\end{equation*}
In other words, $\mathscr{D}^{\alpha\beta}$ transforms co-variantly on $\beta$.
\end{lem}
\begin{proof}
Since both $z$ and $z'$ are morphs for $S_\lambda$, we can define $g = z'\circ z^{-1} :D_\lambda \mapsto D_\lambda$, holomorphic with $g(0)=0$. 

Let $H^\gamma_\beta = \frac{\partial z_\beta}{\partial z'_\gamma} = (g^{-1})'(0)$ be the Jacobian matrix for the coordinate transform at $z=0$. So we can write, $dw_\beta = H^\gamma_\beta dg_\gamma$.

Now,
 \begin{equation*}
 \begin{aligned}
\mathscr{D}^{\alpha\beta} &= \frac{\partial \psi^\alpha_A}{\partial w_\beta} = \frac{\partial \psi^\alpha_A}{\partial g_\gamma}\frac{\partial g_\gamma}{\partial w_\beta} = \mathscr{D}'^{\alpha\gamma}\frac{\partial z'_\gamma}{\partial z_\beta}
 \end{aligned}
 \end{equation*}
\end{proof}

This implies we can talk about $\eta$ and $\mathscr{D}$ in a coordinate-free way upto $\mathscr{Z}$.

\begin{thm}
Let $\mathscr{L}=\{(S_\lambda, \mathscr{S}_\lambda,\nu_\lambda)\}_{\lambda \downarrow 0}$ be an unit lens on a manifold $\mathscr{M}$ and also let $\psi \in \mathscr{P}$ be a meromorphic $\nu_\lambda$-measurable function on $S_\lambda$ with a potential pole of order 1 at $p_0 \in S_\lambda$. If $\eta$ and $\mathscr{D}$ denote the Residue and Jacobian matrix for $\psi$ on $\mathscr{L}$, then 
\begin{enumerate}[label=(\alph*)]
\item{standard error of $\psi$ under the Exterior probability has a lower bound,
\begin{equation*}
\lambda \sigma_{\psi} \geq \sqrt{Tr(\eta^*\eta)}
\end{equation*}
}
\item{
minimum standard error is achieved at a finite scale $*\lambda = \big[\frac{Tr(\eta^*\eta)}{Tr(\mathscr{D}^*\mathscr{D})}\big]^{\frac{1}{4}}$ 
}
\end{enumerate}
\end{thm}
\begin{proof}
\begin{enumerate}[label=(\alph*)]
\item{Let $z\in \mathscr{Z}_\lambda$ and $\psi_z = \psi\circ z^{-1}$. Proof follows as consequence of corollary to Theorem \ref{thm1} as applied on $\psi_z$ and the fact that $\eta$ and $\mathscr{D}$ transforms as tensors under different choice of $z \in \mathscr{Z}_\lambda$
}
\item{
By virtue of Theorem \ref{thm1} applied on $\psi_z$, we can write,
\begin{equation*}
Var_\lambda(\psi) = \frac{Tr(\eta^*\eta)}{\lambda^2} + \lambda^2 Tr(\mathscr{D}^*\mathscr{D})
\end{equation*}
Clearly, this would allow the variance to be minimised for $*\lambda^2 = \sqrt{\frac{Tr(\eta^*\eta)}{Tr(\mathscr{D}^*\mathscr{D})}}$. That leads to $*\lambda= \big[\frac{Tr(\eta^*\eta)}{Tr(\mathscr{D}^*\mathscr{D})}\big]^{\frac{1}{4}}$.
}
\end{enumerate}
\end{proof}

\section{Discussion and Conclusion}
In the preceding sections we have been able to formalise the concept of a Lenses around a point of potential singularity with respect to a particular detectable function, equipped with a specific type of probability measure called Exterior Probability. Under this structure we've seen that the variance of the function, instead of freely dropping to 0 with increasingly closer measurement, has a global lower bound. In other words, too close to the point of singularity, the standard deviation tends to increase in proportion with reducing distance. 

This result was reproduced for a system of lenses over a generic complex manifold, and it was shown that the variance characterisation can be described through tensors in a coordinate free way. The set of results produced in the paper is  most relatable from the perspective of describing quantum scale effects in the backdrop of a generic space-time curved by gravity. The decomposition of variance into parts involving $\eta$ and $\mathscr{D}$ are also significant in that regard, as $\eta$ represents the quantum effects near a singularity, while $\mathscr{D}$ is tied with the generic smooth curvature of the manifold. They can also be thought of representing matter and force fields respectively.

If we consider the special case where $\eta$ is negligible or zero, then we can recover fully the classical measurement set up with the lower bound on standard error converging on zero. On the other hand, if $\mathscr{D}$ is zero, then we have a pure quantum system on a flat space-time. The results show that the standard error for an unbounded measurement attribute in such a system is in accordance with Heisenberg uncertainty principle. 

The other connection that the present work shares with efforts towards understanding quantum gravity, is obviously in the structural similarity between the Lenses and the Strings. While both are posited as sub-structures of space-time and works as a substitution of the concept of classical point mass \cite{gomis2020nonrelativistic}, they differ in a fundamental way as well. Strings are inter-dimensional structures that spans across a sub-space of a higher dimensional space-time. Lenses constructed as they've been in this paper, are intra-dimensional objects plumbing the depth available in each complex dimension for additional structure. Hence Lenses, as opposed to Strings do not require extra large dimensions to produce results consistent with quantum uncertainty.

There are several different directions in which the current results can be progressed or improved further. The results so far are based on the simplest, Unit Lens and associated Exterior probabilities. However, one might be interested to look into more generic lenses and exterior probability structures. Any such probability measures, by definition, would have to be absolutely continuous with the unit lens, and hence would allow a density to be incorporated in the results presented thus far.

Another area of potential interest could be to expand the the set up to attributes with multiple poles instead of just one, poles with higher orders. In physical terms they would be instrumental to understand many-body interactions at quantum scale with strong gravitational backdrop. 

\bibliography{Meromorphic_Variance_Inequality.bib}
\bibliographystyle{unsrt}
\end{document}